\documentclass[a4paper,11pt]{article}
\usepackage{amsmath}
\usepackage{amsfonts}
\usepackage{amssymb}
\usepackage{amsthm}
\usepackage{graphicx}

\date{} 
\begin{document} 
\centerline {\bf {An Element $\displaystyle \phi$-$\delta$-Primary  to another Element in Multiplicative Lattices} }
\centerline{} 

\centerline{\bf {A.~V.~Bingi}}
\centerline{Department of Mathematics}
\centerline{St.~Xavier's College(autonomous), Mumbai-400001, India}
\centerline{$email: ashok.bingi@xaviers.edu$}

\newtheorem{thm}{Theorem}[section]
 \newtheorem{c1}[thm]{Corollary}
 \newtheorem{l1}[thm]{Lemma}
 \newtheorem{prop}[thm]{Proposition}
 \newtheorem{d1}[thm]{Definition}
\newtheorem{rem}[thm]{Remark}
 \newtheorem{e1}[thm]{Example}
\begin{abstract}
  In this paper, we introduce an element $\phi$-$\delta$-primary to another element in a compactly generated multiplicative lattice $L$ and obtain its characterizations. We prove many of its properties and investigate the relations between these structures. By a counter example, it is shown that if an element $b\in L$ is $\phi$-$\delta$-primary to a proper element $p\in L$ then $b$ need not be $\delta$-primary to $p$ and found conditions under which an element $b\in L$ is $\delta$-primary to a proper element $p\in L$ if $b$ is $\phi$-$\delta$-primary to  $p$.
\end{abstract}

{\bf 2010 Mathematics Subject Classification:} 06B99 
\paragraph*{}
{\bf Keywords:-} expansion function,  $\delta$-primary to another element, $\phi$-$\delta$-primary to another element, 2-potent $\delta$-primary to another element, $n$-potent $\delta$-primary to another element, idempotent element

\section{Introduction}
\paragraph*{}  
The notion of an element prime to another element in a multiplicative lattice $L$ is introduced by F.~Alarcon et.~al. in \cite{AAJ}. Further, the notion of an element primary to another element in a multiplicative lattice $L$ is introduced by C.~S.~Manjarekar and Nitin S.~Chavan in \cite{MC}. In an attempt to unify these notions of an element prime to another element and an element primary to another element in a multiplicative lattice $L$ under one frame,  an element $\delta$-primary to another element in a multiplicative lattice $L$ is introduced by Ashok V. Bingi in \cite{A1}. 

Further, the concept of an element weakly prime to another element and  an element weakly primary to another element in a multiplicative lattice $L$ is introduced by C.~S.~Manjarekar and U.~N.~Kandale in \cite{MK}. To generalise these concepts, the study of  an element $\phi$-prime to another element and an element $\phi$-primary to another element in a multiplicative lattice $L$ is done by Ashok V. Bingi in \cite{A2}.  In this paper, we introduce and study, the notion of an element $\phi$-$\delta$-primary to another element in a multiplicative lattice $L$ as a generalization of an element $\delta$-primary to another element in $L$ and unify an element $\phi$-prime to another element and an element $\phi$-primary to another element in $L$, under one frame.

A multiplicative lattice $L$ is a complete lattice provided with commutative, associative and join distributive multiplication in which the largest element $1$ acts as a multiplicative identity. An element $e\in L$ is called meet principal if $a\wedge be=((a:e)\wedge b)e$ for all $a,\ b\in L$. An element $e\in L$ is called join principal if $(ae\vee b):e=(b:e)\vee a$ for all $a,\ b\in L$. An element $e\in L$ is called principal if $e$ is both meet principal and join principal. A multiplicative lattice $L$ is said to be principally generated(PG) if every element of $L$ is a join of principal elements of $L$. An element $a\in L$ is called compact if for  $X\subseteq L$, $a\leqslant \vee X$  implies the existence of a finite number of elements $a_1,a_2,\cdot\cdot\cdot,a_n$ in $X$ such that $a\leqslant a_1\vee a_2\vee\cdot\cdot\cdot\vee a_n$. The set of compact elements of $L$ will be denoted by $L_\ast$. If each element of $L$ is a join of compact elements of $L$ then $L$ is called a compactly generated lattice or simply a CG-lattice. 
       
An element $a\in L$ is said to be proper if $a<1$. The radical of $a\in L$ is denoted by $\sqrt{a}$ and is defined as $\vee \{ x \in L_\ast \mid x^{n} \leqslant a$, for  some  $n\in Z_+\}$. A proper element $m\in L$ is said to be maximal if for every element $x\in L$ such that $m<x\leqslant 1$ implies $x=1$. A proper element $p\in L$ is called a prime element if $ab\leqslant p$ implies $a\leqslant p$ or $b\leqslant p$ where $a,b\in L$ and is called a primary element if $ab\leqslant p$ implies $a\leqslant p$ or $b\leqslant \sqrt{p}$ where $a,b\in L_\ast$. For $a,b\in L $, $(a:b)= \vee \{x \in L \mid xb \leqslant a\} $. A multiplicative lattice is called as a Noether lattice if it is modular, principally generated and satisfies ascending chain condition.  An element $a\in L$ is called a zero divisor if $ab=0$ for some $0\neq b\in L$ and is called idempotent if $a=a^2$. A multiplicative lattice is said to be a domain if it is without zero divisors and is said to be quasi-local if it contains a unique maximal element. A quasi-local multiplicative lattice $L$ with maximal element $m$ is denoted by $(L,\ m)$. A Noether lattice $L$ is local if it contains precisely one maximal prime. In a Noether lattice $L$, an element $a\in L$ is said to satisfy  restricted cancellation law if for all $b,\ c\in L$, $ab=ac\neq 0$ implies $b=c$ (see $\cite{W}$). According to \cite{MB1}, an expansion function on $L$ is a function $\delta:L \longrightarrow L$ which satisfies the following two conditions: \textcircled{1}. $a\leqslant\delta(a)$ for all $a\in L$,
\textcircled{2}. $a\leqslant b$ implies $\delta(a)\leqslant\delta(b)$ for all $a,\ b\in L$.  The reader is referred to \cite{AAJ} for general background and terminology in multiplicative lattices. 

According to \cite{AAJ}, an element $b\in L$ is  said to be prime to a proper element $p\in L$  if $xb\leqslant p$ implies $x\leqslant p$ where $x\in L$. According to \cite{MC}, an element $b\in L$ is  said to be primary to a proper element $p\in L$  if $xb\leqslant p$ implies $x\leqslant \sqrt{p}$ where $x\in L_\ast$. According to \cite{MK}, an element $b\in L$ is  said to be weakly prime to a proper element $p\in L$  if $0\neq xb\leqslant p$ implies $x\leqslant p$ where $x\in L$ and an element $b\in L$ is  said to be weakly primary to a proper element $p\in L$  if $0\neq xb\leqslant p$ implies $x\leqslant \sqrt{p}$ where $x\in L_\ast$.

Further, according to \cite{A1}, given an expansion function $\delta$ on $L$, an element $b\in L$ is said to be  $\delta$-primary to a proper element $p\in L$  if for all $x\in L$, $xb\leqslant p$ implies  $x\leqslant\delta(p)$.  According to \cite{A2}, given a function  $\phi:L\longrightarrow L$, an element $b\in L$ is said to be $\phi$-prime to a proper element $p\in L$ if for all $x\in L$, $xb\leqslant p$ and $xb\nleqslant\phi(p)$ implies $x\leqslant p$ and  an element $b\in L$ is said to be $\phi$-primary to a proper element $p\in L$ if for all $x\in L$, $xb\leqslant p$ and $xb\nleqslant\phi(p)$ implies $x\leqslant \sqrt{p}$.      
             
        In this paper, we define an element $\phi$-$\delta$-primary to another element in $L$ and obtain their characterizations. The notopn of an element $\phi_\alpha$-$\delta$-primary to another element in $L$ is introduced and relations among them are obtained. By counter examples, it is shown that if $b\in L$ is $\phi$-$\delta$-primary to a proper element of $p\in L$ then $b$ need not be $\phi$-prime to $p$, $b$ need not be prime to $p$ and $b$ need not be $\delta$-primary to $p$. In 6 different ways, we have proved if an element $b\in L$ is $\phi$-$\delta$-primary to a proper element $p$ then $b$ is $\delta$-primary to $p$ under certain conditions. We define an element $2$-potent $\delta$-primary to another element of $L$ and an element $n$-potent $\delta$-primary to another element of $L$. Finally, we show that for an idempotent element  $p\in L$, $b\in L$ is $\phi_2$-$\delta$-primary to $p$ but if $b\in L$ is $\phi_2$-$\delta$-primary to a proper element $p\in L$ then $p$ need not be idempotent. Throughout this paper, \textcircled{1}. $L$ denotes a compactly generated multiplicative lattice with $1$ compact in which every finite product of compact elements is compact, \textcircled{2}. $\delta$ denotes an expansion function on $L$ and \textcircled{3}. $\phi$ denotes a function defined on $L$.
         
         \medskip
         
       \section{An element $\displaystyle \phi$-$\delta$-primary to another element in $L$}
       
       \paragraph*{}        
       
           We begin with introducing the notion of  an element of $L$ to be  $\phi$-$\delta$-primary  to another element of $L$ which is the      generalization of the concept of an element to be $\delta$-primary to another element of $L$.

                          \begin{d1}\label{D-C2} 
                           Given an expansion function $\delta:L\longrightarrow L$ and a function $\phi:L\longrightarrow L$, an element $b\in L$ is said to be $\phi$-$\delta$-primary to a proper element $p\in L$ if for all $x\in L$, $xb\leqslant p$ and $xb\nleqslant\phi(p)$ implies  $x\leqslant \delta(p)$.
                          \end{d1}

              For the special functions $\phi_\alpha:L\longrightarrow L$, an element \textbf{``$\phi_\alpha$-$\delta$-primary to"} another element in $L$ is defined by following settings in the definition \ref{D-C2} of an element $\phi$-$\delta$-primary to another element in $L$. For any proper element $p\in L$ in the definition \ref{D-C2}, in place of $\phi(p)$, set
                     \begin{itemize}
                     \item $\phi_0(p)=0$. Then $b\in L$ is called {\bf weakly $\delta$-primary to $p$}.
                     \item $\phi_2(p)=p^2$. Then $b\in L$ is called {\bf $2$-almost $\delta$-primary to $p$} or   {\bf $\phi_2$-$\delta$-primary to $p$} or simply  {\bf almost $\delta$-primary to $p$}.
                     \item $\phi_n(p)=p^n\ (n>2)$. Then $b\in L$ is called {\bf $n$-almost $\delta$-primary to $p$} or  {\bf $\phi_n$-$\delta$-primary to $p$} $(n>2)$.
                     \item $\phi_\omega(p)=\bigwedge_{i=1}^{\infty}p^n$. Then $b\in L$ is called  {\bf $\omega$-$\delta$-primary to $p$} or {\bf $\phi_\omega$-$\delta$-primary to $p$}.
                     \end{itemize} 
                                          
  Since for an element $a\in L$ with $a\leqslant q$ but $a\nleqslant \phi(q)$ implies that $a \nleqslant q\wedge \phi(q)$, there is no loss generality in assuming that $\phi(q) \leqslant q$. We henceforth make this assumption. 
              
     \begin{d1}
     Given any two functions $\gamma_1, \gamma_2 : L\longrightarrow L$, we define $\gamma_1\leqslant \gamma_2$ if $\gamma_1(a)\leqslant \gamma_2(a)$ for each $a\in L$.
     \end{d1} 
     
     Clearly, we have the following order:
     
     \centerline{ $\phi_0 \leqslant \phi_\omega \leqslant \cdots \leqslant \phi_{n+1} \leqslant \phi_n \leqslant \cdots \leqslant \phi_2 \leqslant \phi_1$}
      \medskip
       
        Further as $\phi(p)\leqslant p$ and $p\leqslant \delta(p)$ for each $p\in L$, the relation between the functions $\delta$ and $\phi$ is $\phi\leqslant \delta$.
        
        \medskip  
            
        According to \cite{MB1}, $\delta_0$ is an expansion function on $L$ defined as $\delta_0(p)=p$ for each $p\in L$ and $\delta_1$ is an expansion function on $L$ defined as $\delta_1(p)=\sqrt{p}$ for each $p\in L$. 
       
       The following 2 results relate an element $\phi$-prime to another element and an element $\phi$-primary to another element with some element $\phi$-$\delta$-primary to another element in $L$.
       
       \begin{thm}
       An element $b\in L$ is  $\phi$-$\delta_0$-primary to a proper element $p\in L$  if and
        only if $b$ is  $\phi$-prime to $p$.
       \end{thm}
       \begin{proof}
       The proof is obvious.
       \end{proof}
       
        \begin{thm}
         An element $b\in L$ is  $\phi$-$\delta_1$-primary to a proper element $p\in L$  if and
                only if $b$ is  $\phi$-primary to $p$.
        \end{thm}
        \begin{proof}
        The proof is obvious.
        \end{proof}

       \begin{thm}\label{T-C21}
       Let $\delta,\ \gamma :L \longrightarrow L$ be expansion functions on $L$ such that $\delta \leqslant\gamma$. Let $p\in L$ be a proper element and $b\in L$. If  $b$ is  $\phi$-$\delta$-primary to  $p$ then  $b$ is  $\phi$-$\gamma$-primary to $p$. In particular, for every expansion function $\delta$ on $L$, if  $b$ is  $\phi$-prime to  $p$ then $b$ is $\phi$-$\delta$-primary to $p$.
      \end{thm}
               \begin{proof} Assume that $b\in L$ is  $\phi$-$\delta$-primary to  a proper element $p\in L$. Suppose $xb\leqslant p$ and $xb\nleqslant \phi(p)$ for $x\in L$. Then $x\leqslant \delta(p)\leqslant\gamma(p)$ and so $b$ is  $\phi$-$\gamma$-primary to $p$. Next, for any expansion function $\delta$ on $L$, we have $\delta_0\leqslant \delta$. So if  $b$ is  $\phi$-$\delta_0$-primary to  $p$ then $b$ is $\phi$-$\delta$-primary to $p$ and we are done because if $b$ is  $\phi$-prime to $p$ then  $b$ is $\phi$-$\delta_0$-primary to $p$.
               \end{proof}
               
               \begin{c1}
               For every expansion function $\delta$ on $L$, if an element $b\in L$ is prime to a proper element $p\in L$ then then $b$ is $\phi$-$\delta$-primary to $p$.
               \end{c1}
               \begin{proof}
               The proof follows by using Theorem \ref{T-C21} to the fact that if an element $b\in L$ is prime to a proper element $p\in L$ then then $b$ is $\phi$-prime to $p$.
               \end{proof}

 \medskip   
     
 The following example shows that (by taking $\phi$ as $\phi_2$ and $\delta$ as $\delta_1$ for convenience) 
          \begin{itemize}
          \item If $b\in L$ is $\phi$-$\delta$-primary to a proper element $p\in L$ then $b$ need not be $\phi$-prime to $p$.
                     
          \item If $b\in L$ is $\phi$-$\delta$-primary to a proper element $p\in L$ then $b$ need not be prime to $p$.
          \end{itemize}

         \begin{e1}\label{E-C21}
         Consider the lattice $L$ of ideals of the ring $R=<Z_{24}\ , \ +\ ,\ \cdot>$. Then the only ideals of $R$ are the principal ideals (0),(2),(3),(4),(6),(8),(12),(1). Clearly, $L=\{$(0),(2),(3),(4),(6),(8),(12),(1)$\}$ is a compactly generated multiplicative lattice.
          It is easy to see that the element $(2)\in L$ is $\phi_2$-$\delta_1$-primary to $(4)\in L$ while $(2)$ is not $\phi_2$-prime to $(4)$. Also $(2)$ is not prime to $(4)$.\\        
        \end{e1}
        
        
                 
   Now before obtaining the characterizations of an element $\phi$-$\delta$-primary to another element of $L$, we state the following essential lemma which is outcome of Lemma 2.3.13 from \cite{S0}.
                    
                    \begin{l1}\label{L}
                    Let $a_1,\ a_2\in L$. Suppose $b\in L$ satisfies the following property:
                    
                    ($\ast$). If $h\in L_\ast$ with $h\leqslant b$ then either $h\leqslant a_1$ or $h\leqslant a_2$.
                    
                    Then either $b\leqslant a_1$ or $b\leqslant a_2$.
                    \end{l1}                             
             
         \begin{thm}\label{T-C26}
         Let $p$ be a proper element of $L$ and $b\in L$. Then the following statements are equivalent:
         
         \textcircled{1}. $b$ is $\phi$-$\delta$-primary to $p$. 
         
         \textcircled{2}. either $(p:b)\leqslant \delta(p)$ or $(p:b)=(\phi(p):b)$. 
         
         \textcircled{3}. for every $r\in L_\ast$, $rb\leqslant p$ and $rb\nleqslant \phi(p)$ implies $r\leqslant \delta(p)$. 
         \end{thm}
         \begin{proof}
         \textcircled{1}$\Longrightarrow$\textcircled{2}. Suppose \textcircled{1} holds. Let $h\in L_\ast$ be such that $h\leqslant (p:b)$. Then $hb\leqslant p$. If $hb\leqslant \phi(p)$ then $h\leqslant (\phi(p):b)$. If $hb\nleqslant \phi(p)$ then since $b$ is $\phi$-$\delta$-primary to $p$, $hb\leqslant p$ and $hb\nleqslant \phi(p)$, it follows that $h\leqslant \delta(p)$. Hence by Lemma $\ref{L}$, either $(p:b)\leqslant (\phi(p):b)$ or $(p:b)\leqslant \delta(p)$. Consequently, either $(p:b)=(\phi(p):b)$ or $(p:b)\leqslant \delta(p)$.
         
         \textcircled{2}$\Longrightarrow$\textcircled{3}. Suppose \textcircled{2} holds. Let $rb\leqslant p$ and $rb\nleqslant \phi(p)$ for  $r\in L_\ast$. By \textcircled{2} if $(p:b)=(\phi(p):b)$ then as $r\leqslant (p:b)$, it follows that $r\leqslant (\phi(p):b)$ which contradicts $rb\nleqslant \phi(p)$ and so we must have $(p:b)\leqslant \delta(p)$. Therefore $r\leqslant (p:b)$ gives $r\leqslant \delta(p)$.
         
         \textcircled{3}$\Longrightarrow$\textcircled{1}. Suppose \textcircled{3} holds. Let $xb\leqslant p$ and $xb\nleqslant \phi(p)$ for $x\in L$. Then as $L$ is compactly generated, there exist $y'\in L_\ast$ such that $y'\leqslant x$ and $y'b\nleqslant \phi(p)$. Let $y\leqslant x$ be any compact element of $L$. Then 
         $(y\vee y')\in L_\ast$ such that $(y\vee y')b\leqslant p$ and $(y\vee y')b\nleqslant \phi(p)$. So by \textcircled{3}, it follows that $(y\vee y')\leqslant \delta(p)$ which implies $x\leqslant \delta(p)$ and therefore $b$ is $\phi$-$\delta$-primary to $p$.
         \end{proof} 
         
\begin{thm}
 Let $(L,\ m)$ be a quasi-local Noether lattice. If a proper element $p\in L$ is such that $p^2=m^2\leqslant p\leqslant m$ and $b\in L$ then $b$ is either $\phi_2$-$\delta_1$-primary to $p$ or $b\leqslant p$.
 \end{thm}
 \begin{proof}
 Let $xb\leqslant p$ and $xb\nleqslant \phi_2(p)$ for $x\in L$. If $x\nleqslant m$ then $x=1$. So $xb\leqslant p$ gives $b\leqslant p$.  Now if $x\leqslant m$ then $x^2\leqslant m^2=p^2\leqslant p$ and hence $x\leqslant\delta_1(p)$ which implies $b$ is $\phi_2$-$\delta_1$-primary to $p$.  Thus $b$ is either $\phi_2$-$\delta_1$-primary to $p$ or $b\leqslant p$.
 \end{proof}                  
             
      To obtain the relation among an element $\phi_\alpha$-$\delta$-primary to another element in $L$, we prove the following lemma.
             \begin{l1}\label{L-C1}
              Let $\gamma_1,\ \gamma_2 :L \longrightarrow L$ be functions on $L$ such that $\gamma_1 \leqslant \gamma_2$ and p be proper element of $L$. If an element $b\in L$ is $\gamma_1$-$\delta$-primary to p then $b\in L$ is $\gamma_2$-$\delta$-primary to p. 
             \end{l1}
                      \begin{proof} Let an element $b\in L$ be $\gamma_1$-$\delta$-primary to p. Suppose $xb\leqslant p$ and $xb\nleqslant \gamma_2(p)$ for $x\in L$. Then as $\gamma_1\leqslant\gamma_2$, we have $xb\leqslant p$ and $xb\nleqslant \gamma_1(p)$. Since $b$ is $\gamma_1$-$\delta$-primary to p, it follows that $x\leqslant\delta(p)$ and hence $b$ is $\gamma_2$-$\delta$-primary to p.
                      \end{proof}    
                         
     \begin{thm}\label{T-C22}
             For an element $b\in L$ and a proper element $p\in L$, consider the following statements:
             \begin{enumerate}
             \item[(a)] $b$ is $\delta$-primary to $p$. 
             \item[(b)] $b$ is $\phi_0$-$\delta$-primary to $p$.
             \item[(c)] $b$ is $\phi_\omega$-$\delta$-primary to $p$.
             \item[(d)] $b$ is $\phi_{(n+1)}$-$\delta$-primary to $p$.
             \item[(e)] $b$ is $\phi_n$-$\delta$-primary to $p$ where $n\geqslant 2$.
             \item[(f)] $b$ is $\phi_2$-$\delta$-primary to $p$.
             \end{enumerate}
             Then $(a) \Longrightarrow (b) \Longrightarrow (c) \Longrightarrow (d) \Longrightarrow (e) \Longrightarrow (f)$.
           \end{thm}
          \begin{proof}
          Obviously, if $b$ is $\delta$-primary to $p$ then $b$ is weakly $\delta$-primary to $p$ and hence $(a) \Longrightarrow (b)$. The remaining implications follow by using Lemma \ref{L-C1} to the fact that $\phi_0 \leqslant \phi_\omega \leqslant \cdots \leqslant \phi_{n+1} \leqslant \phi_n \leqslant \cdots \leqslant \phi_2$
          \end{proof}        
              
     \begin{c1}\label{L-C21}
       Let $p\in L$ be a proper element and $b\in L$. Then $b$ is $\phi_\omega$-$\delta$-primary to $p$ if and only if $b$ is $\phi_n$-$\delta$-primary to $p$ for every $n\geqslant 2$.
       \end{c1}
       \begin{proof}
       Assume that $b$ is $\phi_n$-$\delta$-primary to $p$  for every $n\geqslant 2$. Let $xb\leqslant p$ and $xb\nleqslant\bigwedge_{n=1}^{\infty}p^n$ for $x\in L$. Then $xb\leqslant p$ and $xb\nleqslant p^n$ for some $n\geqslant 2$. Since $b$ is $\phi_n$-$\delta$-primary to $p$, we have $x\leqslant\delta(p)$ and hence $b$ is $\phi_\omega$-$\delta$-primary to $p$. The converse follows from Theorem $\ref{T-C22}$.
       \end{proof}
       
   Now we show that under a certain condition, if an element $b\in L$ is $\phi_n$-$\delta$-primary $(n\geqslant 2)$ to a proper element $p\in L$ then $b$ is $\delta$-primary to $p$.
     
     \begin{thm}\label{T-C74}
     Let $L$ be a local Noetherian domain. Let $p\in L$ be a proper element and $0\neq b\in L$.  Then $b$ is $\phi_n$-$\delta$-primary to $p$ for every $n\geqslant 2$ if and only if $b$ is $\delta$-primary to $p$.
     \end{thm}
     \begin{proof}
     Assume that $b$ is $\phi_n$-$\delta$-primary to $p$ for every $n\geqslant 2$. Let $xb\leqslant p$ for $x\in L$. If $xb\nleqslant \phi_n(p)$ for $n\geqslant 2$ then as $b$ is $\phi_n$-$\delta$-primary to $p$, we have $x\leqslant \delta(p)$. If $xb\leqslant \phi_n(p)=p^n$ for all $n\geqslant 1$ then as $L$ is local Noetherian, by Corollary 3.3 of $\cite{D3}$, it follows that $xb\leqslant\bigwedge_{n=1}^{\infty}p^n=0$ and so $xb=0$. Since $L$ is domain and $0\neq b$, we have $x=0$  which implies $x\leqslant \delta(p)$. Hence, in any case, $b$ is $\delta$-primary to $p$. Converse follows from Theorem \ref{T-C22}.
     \end{proof}
     
     \begin{c1}
     Let $L$ be a local Noetherian domain. Let $p\in L$ be a proper element and $0\neq b\in L$. Then $b$ is $\phi_\omega$-$\delta$-primary to $p$ if and only if $b$ is $\delta$-primary to $p$ .
     \end{c1}
     \begin{proof}
     The proof follows from Theorem $\ref{T-C74}$ and Corollary $\ref{L-C21}$.
     \end{proof}
          
 Clearly, if an element $b\in L$ is $\delta$-primary to a proper element $p\in L$ then $b$ is $\phi$-$\delta$-primary to $p$. The following example shows that its converse is not true (by taking $\phi$ as $\phi_2$ and $\delta$ as $\delta_1$ for convenience).
            
            \begin{e1}\label{E-C22}
            Consider the lattice $L$ of ideals of the ring $R=<Z_{30}\ , \ +\ ,\ \cdot>$. Then the only ideals of $R$ are the principal ideals (0),(2),(3),(5),(6),(10),(15),(1). Clearly $L=\{$(0),(2),(3),(5),(6),(10),(15),(1)$\}$ is a compactly generated multiplicative lattice.
             It is easy to see that the element $(2)\in L$ is  $\phi_2$-$\delta_1$-primary to $(6)\in L$ but $(2)$ is not $\delta_1$-primary to $(6)$.\\
            \end{e1} 
   
   In the following successive six theorems, we show conditions under which if an element $b\in L$ is $\phi$-$\delta$-primary to a proper element $p$ then $b$ is $\delta$-primary to $p$. 
   
   \begin{thm}\label{T-C72}
   Let $L$ be a Noether lattice. Let $0\neq p\in L$ be a non-nilpotent proper element satisfying the restricted cancellation law. Let $b\in L$ be such that $p<b$. Then $b$ is $\phi$-$\delta$-primary to $p$ for some $\phi\leqslant\phi_2$ if and only if $b$ is $\delta$-primary to $p$.
   \end{thm} 
   \begin{proof}
   Assume that $b\in L$ is $\delta$-primary to $p\in L$. Then obviously, $b$ is $\phi$-$\delta$-primary to $p$ for every $\phi$ and hence for some $\phi\leqslant \phi_2$. Conversely, assume that $b$ is $\phi$-$\delta$-primary to $p$ for some $\phi\leqslant\phi_2$. Then by Lemma \ref{L-C1}, $b$ is $\phi_2$-$\delta$-primary (almost $\delta$-primary) to $p$. Let $xb\leqslant p$ for $x\in L$. If $xb\nleqslant \phi_2(p)$ then as $b$ is $\phi_2$-$\delta$-primary to $p$, we have $x\leqslant \delta(p)$. If $xb\leqslant \phi_2(p)=p^2$ then $xp\leqslant p^2\neq 0$ as $p<b$.  Hence $x\leqslant p\leqslant \delta(p)$ by Lemma 1.11 of $\cite{W}$ and thus $b$ is $\delta$-primary to $p$.
   \end{proof}
   
  \begin{c1}
   Let $L$ be a Noether lattice. Let $0\neq p\in L$ be a non-nilpotent proper element satisfying the restricted cancellation law. Let $b\in L$ be such that $p<b$.  If $b$ is $\phi_2$-$\delta$-primary to $p$  then $b$ is $\delta$-primary to $p$.
  \end{c1}
   \begin{proof}
   The proof follows from proof of the Theorem \ref{T-C72}.
   \end{proof}
   
   The following result is general form of Theorem \ref{T-C72}.
   
   \begin{thm}\label{T-C79}
      Let $L$ be a Noether lattice. Let $0\neq p\in L$ be a non-nilpotent proper element satisfying the restricted cancellation law. Let $b\in L$ be such that $p<b$. Then $b$ is $\phi$-$\delta$-primary to $p$ for some $\phi\leqslant\phi_n$ and for all $n\geqslant 2$ if and only if $b$ is $\delta$-primary to $p$. 
    \end{thm} 
     \begin{proof}
     Assume that $b$ is $\delta$-primary to $p$. Then obviously, $b$ is $\phi$-$\delta$-primary to $p$ for every $\phi$ and hence for some $\phi\leqslant \phi_n$, for all $n\geqslant 2$. Conversely, assume that $b$ is $\phi$-$\delta$-primary to $p$ for some $\phi\leqslant\phi_n$ and for all $n\geqslant 2$. Then by Lemma \ref{L-C1}, $b$ is $\phi_n$-$\delta$-primary ($n$-almost $\delta$-primary) to $p$ and for all $n\geqslant 2$. Let $xb\leqslant p$ for $x\in L$. If $xb\nleqslant \phi_n(p)$ for some $n\geqslant 2$ then as $b$ is $\phi_n$-$\delta$-primary to $p$, we have $x\leqslant \delta(p)$ and we are done. So let $xb\leqslant \phi_n(p)$ for all $n\geqslant 2$. 
       Then $xb\leqslant p^n\leqslant p^2$ as $n\geqslant 2$. This implies $xp\leqslant p^2\neq 0$ as $p<b$.  Hence $x\leqslant p\leqslant \delta(p)$ by Lemma 1.11 of $\cite{W}$ and thus $b$ is $\delta$-primary to $p$.        
     \end{proof}
      
     \begin{c1}
     Let $L$ be a Noether lattice. Let $0\neq p\in L$ be a non-nilpotent proper element satisfying the restricted cancellation law. Let $b\in L$ be such that $p<b$.  If $b$ is $\phi_n$-$\delta$-primary to $p$ $(\forall$  $n\geqslant 2)$  then $b$ is $\delta$-primary to $p$.     
     \end{c1}
     \begin{proof}
      The proof follows from proof of the Theorem $\ref{T-C79}$.
      \end{proof}
      
  \begin{d1} 
   An element $b\in L$ is said to be {\bf 2-potent $\delta$-primary} to a proper element $p\in L$ if for all $x\in L$, $xb\leqslant p^2$ implies  $x\leqslant\delta(p)$.
   \end{d1}

   \begin{thm}\label{T-C80}
   Let $b\in L$ be 2-potent $\delta$-primary to a proper element $p\in L$. Then $b$ is $\phi$-$\delta$-primary to $p$ for some $\phi\leqslant\phi_2$ if and only if $b$ is $\delta$-primary to $p$.
      \end{thm} 
      \begin{proof}
      Assume that $b$ is $\delta$-primary to $p$. Then obviously, $b$ is $\phi$-$\delta$-primary to $p$ for every $\phi$ and hence for some $\phi\leqslant \phi_2$. Conversely, assume that $b$ is $\phi$-$\delta$-primary to $p$ for some $\phi\leqslant\phi_2$. Then by Lemma \ref{L-C1}, $b$ is $\phi_2$-$\delta$-primary (almost $\delta$-primary) to $p$. Let $xb\leqslant p$ for $x\in L$. If $xb\nleqslant \phi_2(p)$ then as $b$ is $\phi_2$-$\delta$-primary to $p$, we have $x\leqslant \delta(p)$. If $xb\leqslant \phi_2(p)=p^2$ then as $b$ is 2-potent $\delta$-primary to $p$, we have $x\leqslant\delta(p)$. Hence $b$ is $\delta$-primary to $p$.
   \end{proof}
   
   \begin{c1}
   Let $p\in L$ be a proper element and $b\in L$.
   If $b$ is $\phi_2$-$\delta$-primary to $p$ and $b$ is 2-potent $\delta$-primary to  $p$ then $b$ is $\delta$-primary to $p$.
   \end{c1}
   \begin{proof}
   The proof follows from proof of the Theorem $\ref{T-C80}$.
   \end{proof}
   
    \begin{d1} 
      Let $n\geqslant 2$. An element $b\in L$ is said to be {\bf $n$-potent $\delta$-primary} to a proper element $p\in L$ if for all $x\in L$, $xb\leqslant p^n$ implies  $x\leqslant\delta(p)$. 
      \end{d1} 
      
 Obviously, if an element $b\in L$ is $n$-potent $\delta_0$-primary to a proper element $p\in L$ then $b$ is 2-potent $\delta$-primary to $p$.\\    
      
 The following result is general form of Theorem \ref{T-C80}.     
      
     \begin{thm}\label{T-C3}
      Let $p\in L$ be a proper element and $b\in L$. Then $b$ is $\phi$-$\delta$-primary to $p$ for some $\phi\leqslant\phi_n$ where $n\geqslant 2$ if and only if $b$ is $\delta$-primary to $p$, provided $b$ is $k$-potent $\delta$-primary to $p$ for some $k\leqslant n$.
      \end{thm} 
      \begin{proof}
      Assume that $b$ is $\delta$-primary to $p$. Then obviously, $b$ is $\phi$-$\delta$-primary to $p$ for every $\phi$ and hence for some $\phi\leqslant \phi_n$ where $n\geqslant 2$. Conversely, assume that $b$ is $\phi$-$\delta$-primary to $p$ for some $\phi\leqslant\phi_n$ where $n\geqslant 2$. Then by Lemma \ref{L-C1}, $b$ is $\phi_n$-$\delta$-primary ($n$-almost $\delta$-primary) to $p$. Let $xb\leqslant p$ for $x\in L$. If $xb\nleqslant \phi_k(p)=p^k$ then $xb\nleqslant \phi_n(p)=p^n$ as $k\leqslant n$. Since $b$ is $\phi_n$-$\delta$-primary to $p$, we have $x\leqslant \delta(p)$. If $xb\leqslant \phi_k(p)=p^k$ then as $b$ is $k$-potent $\delta$-primary to $p$, we have $x\leqslant\delta(p)$. Hence $b$ is $\delta$-primary to $p$.
      \end{proof}  
      
   \begin{c1}
    Let $p\in L$ be a proper element and $b\in L$.
      If $b$ is $\phi_n$-$\delta$-primary to $p$ and $b$ is $k$-potent $\delta$-primary to  $p$ where $k\leqslant n$ then $b$ is $\delta$-primary to $p$.
         \end{c1}    
      
  \begin{thm}\label{T-C23}
   Let $p\in L$ be a proper element and $b\in L$ be $\phi$-$\delta$-primary to $p$. If $pb\nleqslant \phi(p)$ then $b$ is $\delta$-primary to $p$.
   \end{thm}
   \begin{proof}
   Let $xb\leqslant p$ for $x\in L$. If $xb\nleqslant \phi(p)$ then as $b$ is $\phi$-$\delta$-primary to $p$, we have $x\leqslant \delta(p)$. So assume that $xb\leqslant\phi(p)$. Then as $pb\nleqslant\phi(p)$, we have $db\nleqslant\phi(p)$ for some $d\leqslant p$ in $L$. Also $(x\vee d)b=xb\vee db\leqslant p$ and $(x\vee d)b\nleqslant\phi(p)$. As $b$ is $\phi$-$\delta$-primary to $p$, we have $x\leqslant (x\vee d)\leqslant \delta(p)$ and hence $b$ is $\delta$-primary to $p$.
   \end{proof}
   
   From the Theorem $\ref{T-C23}$, it follows that,
    if an element $b\in L$ is $\phi$-$\delta$-primary to a proper element $p\in L$  but $b$ is not $\delta$-primary to $p$ then $pb\leqslant \phi(p)$ and hence $pb\leqslant p$.

   \begin{c1}
    If an element $b\in L$ is $\phi_0$-$\delta$-primary to a proper element $p\in L$  but $b$ is not $\delta$-primary to $p$ then $pb=0$.
   \end{c1}
   \begin{proof}
   The proof is obvious.
   \end{proof}
   
   \begin{thm}
   Let an element $b\in L$ be $\phi$-$\delta$-primary to a proper element $p\in L$. If $b$ is $\delta$-primary to $\phi(p)$ then $b$ is $\delta$-primary to $p$.
   \end{thm}
   \begin{proof}
   Let $xb\leqslant p$ for $x\in L$. If $xb\nleqslant \phi(p)$ then as $b$ is $\phi$-$\delta$-primary to $p$, we have  $x\leqslant \delta(p)$ and we are done. Now if $xb\leqslant \phi(p)$ then as $b$ is $\delta$-primary to $\phi(p)$, we have $x\leqslant \delta(\phi(p))$. This implies that  $x\leqslant \delta(p)$ because $\phi(p)\leqslant p$ and we are done. 
   \end{proof}
   
       
       
    The following theorem shows that a under certain condition, $b\in L$ is $\phi$-$\delta$-primary to $(p:q)\in L$   if $b$ is $\phi$-$\delta$-primary to  $p\in L$  where $q\in L$.
       
       \begin{thm}
      Let an element $b\in L$ be $\phi$-$\delta$-primary to a proper element $p\in L$. Then $b$ is  $\phi$-$\delta$-primary to $(p:q)$ for all $q\in L$ if $(\phi(p):q)\leqslant \phi(p:q)$ and $(\delta(p):q)\leqslant \delta(p:q)$.
       \end{thm}       
       \begin{proof}
       Let $xb\leqslant (p:q)$ and $xb\nleqslant \phi(p:q)$  for $x\in L$. Then $xqb\leqslant p$ and $xqb\nleqslant \phi(p)$. Now as $b$ is  $\phi$-$\delta$-primary to $p$, we have $xq\leqslant \delta(p)$ which implies $x\leqslant (\delta(p):q)\leqslant \delta(p:q)$ and hence  $b$ is  $\phi$-$\delta$-primary to $(p:q)$. 
       \end{proof} 
       
                     \begin{thm}
                     If an element $b^k\in L$ is $\phi$-$\delta_1$-primary to a proper element $p\in L$ for all $k\in Z_+$ such that $\delta_1(\phi(p))=\phi(\delta_1(p))$ then $b$ is  $\phi$-prime to $\delta_1(p)$ where $b\in L$.
                     \end{thm}
                     \begin{proof} Assume that $xb\leqslant \delta_1(p)$ and $xb\nleqslant\phi(\delta_1(p))$ for $x\in L$. Then there exists $n\in Z_+$ such that $x^n\cdot b^n=(xb)^n\leqslant p$. If $(xb)^n\leqslant\phi(p)$ then by hypothesis $xb\leqslant \delta_1(\phi(p))=\phi(\delta_1(p))$, a contradiction. So we must have $x^n\cdot b^n=(xb)^n\nleqslant\phi(p)$. Since  $b^n$ is $\phi$-$\delta_1$-primary to $p$ we have, $x^n\leqslant\delta_1(p)$ and hence $x\leqslant  \delta_1{(\delta_1(p))}=\delta_1(p)$. This shows that $b$ is  $\phi$-prime to $\delta_1(p)$.
                     \end{proof}
                        
   Now we relate idempotent element of $L$ with an element $\phi_n$-$\delta$-primary  $(n\geqslant 2)$ to another element of $L$.
       
       \begin{thm}\label{T-C71}
       If $p$ is an idempotent element of $L$  then $b\in L$ is  $\phi_\omega$-$\delta$-primary to $p$ and hence  $b$ is  $\phi_n$-$\delta$-primary $(n\geqslant 2)$ to $p$.
       \end{thm}
       \begin{proof}
       As $p$ is an idempotent element of $L$, we have $p=p^n$ for all $n\in Z_+$. So $\phi_\omega(p)=p$. Therefore $b$ is  $\phi_\omega$-$\delta$-primary to $p$. Hence $b$ is  $\phi_n$-$\delta$-primary $(n\geqslant 2)$ to $p$, by Theorem \ref{T-C22}.
       \end{proof}
       
       As a consequence of Theorem \ref{T-C71}, we have following result whose proof is obvious.
       
       \begin{c1}
       If $p$ is an idempotent element of $L$  then $b\in L$ is  $\phi_2$-$\delta$-primary to $p$. 
       \end{c1}
       
       However, if $b\in L$ is  $\phi_2$-$\delta$-primary to $p\in L$ then $p$ need not be idempotent as shown in the following example (by taking $\delta$ as $\delta_1$ for convenience).
       
       \begin{e1}\label{E-C23}
       Consider the lattice $L$ of ideals of the ring $R=<Z_8\ , \ +\ ,\ \cdot>$. Then the only ideals of $R$ are the principal ideals (0),(2),(4),(1). Clearly, $L=\{(0),(2),(4),(1)\}$ is a compactly generated multiplicative lattice.
         It is easy to see that the element $(2)\in L$ is $\phi_2$-$\delta_1$-primary to $(4)\in L$ but $(4)$ is not idempotent.
       
     
       \end{e1}
       
       We conclude this paper with the following examples, from which it is clear that,
       \begin{itemize}
                 \item If $b\in L$ is $\phi_2$-$\delta_1$-primary to $p\in L$ then $b$ need not be 2-potent $\delta_0$-primary to $p$.
        
       \item If $b\in L$ is 2-potent $\delta_0$-primary to $p\in L$ and  $b$ is $\phi_2$-$\delta_1$-primary to $p$ then  $b$ need not be prime to $p$.
       \end{itemize}
       
       \begin{e1}
       Consider $L$ as in Example $\ref{E-C22}$. Here the element $(3)\in L$ is $\phi_2$-$\delta_1$-primary to $(6)\in L$ but $(3)$ is not $2$-potent $\delta_0$-primary to $(6)$.
       \end{e1}
       
       \begin{e1}
       Consider $L$ as in Example $\ref{E-C23}$. Here the element $(2)\in L$ is $\phi_2$-$\delta_1$-primary to $(4)\in L$ and $(2)$ is $2$-potent $\delta_0$-primary to $(4)$ but $(2)$ is not prime to $(4)$.
       \end{e1}

\end{document}